\documentclass{amsart}
\usepackage{amssymb}
\newtheorem{theorem}{Theorem}[section]
\newtheorem{proposition}[theorem]{Proposition}
\newtheorem{corollary}[theorem]{Corollary}
\newtheorem{lemma}[theorem]{Lemma}

\theoremstyle{definition}

\theoremstyle{remark}

\newtheorem{claim}[theorem]{Claim}

\numberwithin{equation}{section}

\begin{document}

\title[Polar actions on certain principal bundles]{Polar actions on certain principal bundles \\
over symmetric spaces of compact type}
\author{Marco Mucha}
\address{ Instituto de Matem\'atica e Estat\'istica, Universidade de S\~ao Paulo, Rua do Mat\~ao, 1010 S\~ao Paulo, SP 05508-090, Brazil}
\email{mmuchao@ime.usp.br}
\thanks{This research was supported by FAPESP grant 2007/59288-2.}
\date{}

\commby{Chuu-Lian Terng}
\begin{abstract}
We study polar actions with horizontal sections on the total space of 
certain principal bundles $G/K\to G/H$ with base a symmetric space of 
compact type. We classify such actions up to orbit equivalence in many cases. In particular, 
we exhibit examples of hyperpolar actions with cohomogeneity 
greater than one on locally irreducible homogeneous spaces with nonnegative 
curvature which are not homeomorphic to symmetric spaces. 
\end{abstract}

\maketitle

\section{Introduction}

An isometric action of a compact Lie group on a Riemannian manifold $M$
is called \emph{polar} if there exists an isometrically immersed connected complete
submanifold $i:\Sigma\to M$ which meets every orbit and always orthogonally; any
such $\Sigma$ is called a \emph{section}, and if the section is flat, the action is
called \emph{hyperpolar}. Note that the immersion $i$ might not be injective. For
an introduction to polar actions, see~\cite{BCO,PaTe}. Most of the known
examples of polar actions are polar representations~\cite{Da}, or actions on
symmetric spaces~\cite{Go,Ko2001,Ko2007,PoTh1999}. In this paper, we propose to
consider polar actions on a certain class of nonsymmetric, homogeneous spaces.

Let $G/H$ be a symmetric space of compact type endowed with a Riemannian metric 
induced from some $\mathrm{Ad}(G)$-invariant inner product $g$ 
on the Lie algebra $\mathfrak g$ of $G$. Write $\mathfrak g=\mathfrak h +\mathfrak p_1$
for the Cartan decomposition. We assume that $H$ is connected, and
$\mathfrak h=\mathfrak k\oplus\mathfrak p_2$
(direct sum of ideals of $\mathfrak h$). Let $K$ denote the connected Lie
subgroup of $G$ with Lie algebra $\mathfrak k$.
The natural projection $\pi:G/K\rightarrow G/H$ is an
equivariant principal $H/K$-bundle with respect to the natural projection
$\pi_G:G\times H/K\rightarrow G$. Define $\mathfrak q=\mathfrak p_1\oplus \mathfrak p_2$
and consider the $\mathrm{Ad}(K)$-invariant splitting
$\mathfrak g=\mathfrak k+\mathfrak q$. For each $s>0$, we define an
$\textrm{Ad}(K)$-invariant inner product $g_s$ on
$\mathfrak q$ by $g_s=g|_{\mathfrak p_1\times \mathfrak p_1}+
s^2g|_{\mathfrak p_2\times\mathfrak p_2}.$
We also denote by $g_s$ the induced $G$-invariant metric on $G/K$. It is easy to show
that the action of $G\times H/K$ on  $(G/K,g_s)$ defined by $(g,hK).g'K=gg'h^{-1}K$
is almost effective and isometric. If $L$ is a Lie subgroup of $G$, then the induced
action of $L\times H/K$ on $(G/K,g_s)$ is called the \emph{natural lifting} of the action
of $L$ on $G/H$. Our main result can be stated as follows.
(Recall that two isometric actions are called \emph{orbit equivalent} 
if there exists an isometry between the target spaces which maps orbits 
onto orbits.) 

\begin{theorem}
\label{th:main}
Let $L$ be a closed Lie subgroup of $G$ which acts polarly on $G/H$, and let
$\Sigma$ be a section containing $1H$. We identify $\mathrm{T}_{1H}\Sigma=\mathfrak m\subset \mathfrak p_1$ as usual. If $[\mathfrak m,\mathfrak m]\subset \mathfrak k$,
then $L\times H/K$ acts polarly on $(G/K,g_s)$ with sections horizontal
with respect to $\pi$. Conversely, every polar action on 
$(G/K,g_s)$ of a connected closed Lie subgroup of $G\times H/K$ with horizontal sections
is orbit equivalent to the natural lifting of a polar action on~$G/H$.
\end{theorem}

We remark that there exist many polar actions on $G/H$ satisfying 
the assumption on $\mathfrak m$ in Theorem~\ref{th:main} (for some $\mathfrak k$).
For example, this assumption is satisfied by hyperpolar actions (here $[\mathfrak m,\mathfrak m]=0$), and arbitrary polar actions if $G/H$ is either an irreducible Hermitian symmetric space or a Wolf space. In fact, sections of polar actions on the last two classes of spaces are
totally real with respect to the complex and quaternionic structure respectively;
see~\cite{PoTh1999,Te}.

In~\cite{Zi}, Ziller shows (in a more general context) that the metric $g_s$ is
$G\times H/K$-naturally reductive for all $s>0$, and $G\times H/K$-normal homogeneous if $G$ is simple and $s<1$. On the other hand, we have that
$(G/K,g_s)$ is locally irreducible if $G$ is simple; see section~\ref{sec:4}.
Hence, we get examples of hyperpolar actions with cohomogeneity
greater than one on locally irreducible naturally reductive spaces with nonnegative
curvature which are not homeomorphic to symmetric spaces; see section~\ref{sec:5}.

In many cases, we can use results of~\cite{OlRe,Re} to find the identity component of the isometry group of $(G/K,g_s)$; see Propositions~\ref{pro:iso1} and~\ref{pro:iso2}. In particular we have:
\begin{corollary}
\label{cor:1}
Let $G$ be a simple Lie group. We assume that either $s\neq1$, or $s=1$ and $K$ is nontrivial. If $(G/K,g_s)$ is neither isometric to a round sphere, nor to a real projective space, then every polar action on $(G/K,g_s)$ with horizontal sections is orbit equivalent to the natural lifting of a polar action on $G/H$.
\end{corollary}
In section~\ref{sec:2}, we recall some properties of $(G/K,g_s)$.
The main theorem is proved in section~\ref{sec:3}. Section~\ref{sec:4} is
devoted to proving that $(G/K,g_s)$ is locally irreducible and to computing the identity
component of its isometry group. Finally, we exhibit some examples in
section~\ref{sec:5}.

\section{Preliminaries}
\label{sec:2}

In this paper, we always refer to the notation in the introduction. Next, we recall some results of Ziller~\cite{Zi}.

Another presentation of $G/K$ is obtained from the transitive action of $\bar G:=G\times H/K$ on $G/K$. Let $\bar K$ be the corresponding isotropy group at the basepoint. If $\bar{\mathfrak g}$ and $\bar{\mathfrak k}$ are the Lie algebras of $\bar G$ and $\bar K$ respectively, then
\begin{align*} 
\bar{\mathfrak{g}}& = \mathfrak{k}\oplus\mathfrak{p}_1\oplus\mathfrak p_2\oplus\mathfrak{h}/\mathfrak{k},\\
\bar{\mathfrak{k}}& = \mathfrak{k}\oplus\{\, (0,X,X+\mathfrak{k})\in \mathfrak{p}_1\oplus\mathfrak p_2\oplus\mathfrak{h}/\mathfrak{k} \mid X\in \mathfrak p_2\, \}.
\end{align*}
As a reductive complement, we can take $\bar{\mathfrak q}=\mathfrak p_1\oplus\{\, (0,s^2X,(s^2-1)X+\mathfrak k)\in \mathfrak p_1\oplus\mathfrak p_2\oplus\mathfrak h/\mathfrak k \mid X\in \mathfrak p_2 \,\}$. 
The isomorphism between $\bar G/\bar K$ and $G/K$ on the Lie algebra level sends $\mathfrak{p_1}$ to $\mathfrak{p_1}$ as $id$ and $(0,s^2X,(s^2-1)X+\mathfrak k)$ to $X$,
so the metric $g_s$ looks as follows on $\bar{\mathfrak q}$:
\begin{align*} 
g_s|_{\mathfrak p_1\times \mathfrak p_1}& = \textrm{as before,} \\ 
g_s(\mathfrak p_1,(0,s^2X,(s^2-1)X+\mathfrak k))& = 0, \\
g_s((0,s^2X,(s^2-1)X),(0,s^2Y,(s^2-1)Y+\mathfrak k))& = s^2g(X,Y). \\ 
\end{align*}

The metric $g_s$ is $G\times H/K$-naturally reductive  with respect to the decomposition $\mathfrak{\bar g}=\mathfrak{\bar k}+
\mathfrak{\bar q}$; see~\cite[Theorem $3$]{Zi}. Moreover, if $G$ is simple, $g_s$ is not $G$-normal homogeneous for $s\neq 1$, and $g_s$ is $G\times H/K$-normal homogeneous iff $s<1$. In particular $g_s$  is nonnegatively curved if $G$ is simple and $s<1$. 
\section{Proof of the main theorem}
\label{sec:3}
Let $L$ be a closed Lie subgroup of $G$ which acts polarly on $G/H$, and let $\Sigma$ be a section containing $1H$. Let $\mathfrak s:=\mathfrak m+[\mathfrak m,\mathfrak m]$, where $\mathrm T_{1H}\Sigma=\mathfrak m\subset \mathfrak p_1$.
Then $\mathfrak s$ is a Lie subalgebra of $\mathfrak g$.
Let $S$ be the connected Lie subgroup of $G$ with Lie algebra $\mathfrak s$.
Then $\Sigma=S(1H)$ is a symmetric space and homogeneous
under $S$. We prove that $S(1K)$ is a section for the action of $\bar L:=L\times H/K$
on $(G/K,g_s)$. Since
$\pi^{-1}(gH)\subset \bar L(gK)$ and $L(gH)$ meets $S(1H)$,
$\bar L(gK)$ meets $S(1K)$ for all $gK\in G/K$. Recall that $G\times H/K$
preserves the vertical distribution (and then the horizontal distribution, with respect to $\pi$) on $(G/K,g_s)$. The hypothesis
$[\mathfrak m,\mathfrak m]\subset \mathfrak k$ implies that
$\mathrm{T}_{1K}S(1K)=\mathfrak m$, and hence $S(1K)$ is horizontal.
It follows that the $\bar L$-orbits in $G/K$ meet $S(1K)$ always orthogonally since $\pi:(G/K,g_s)\to (G/H,g)$ is a Riemannian submersion.

Conversely, suppose that $\bar L$ is a connected closed
Lie subgroup of  $G\times H/K$ which acts polarly on $(G/K,g_s)$ with a horizontal section $\bar\Sigma$.
\begin{claim}
$1\times H/K(x)\subset \bar L(x)$, for any $\bar L$-regular point $x\in G/K$.
\end{claim}
\begin{proof}
In fact, let $M_r\subset G/K$ be the set of all the $\bar L$-regular points and let $\mathcal V$ be the vertical distribution on $G/K$. 
Then $\mathcal V$ is integrable with leaves given by the $1\times H/K$-orbits.
Let $x\in M_r$.  Then there exists $l\in\bar L$ such that $lx\in \bar\Sigma$. Since $\bar \Sigma$
is horizontal and $x$ is a regular point, $\mathcal V _{lx} \subset \mathrm{T}_{lx}\bar L(lx)$.
It follows that $\mathcal V_x\subset \mathrm{T}_x{\bar L(x)}$ since $\mathcal V$ is $\bar L$-invariant. Now it is clear that $1\times H/K(x)\subset \bar L(x)$. 
\end{proof}
\begin{claim}
Let $L':=\pi_G(\bar L)\times H/K$. Then the actions of $L'$ and $\bar L$ on $(G/K,g_s)$ are orbit-equivalent.
\end{claim}
\begin{proof}
In fact, we have that $\bar L(x)\subset L'(x)$ for all $x\in G/K$ since $\bar L\subset L'$. Conversely,
the previous claim implies that $L'(x)\subset \bar L(x)$ for all $\bar L$-regular points $x\in G/K$.
This already implies that the orbits of $\bar L$ and $L'$ in $G/K$ coincide (see e.g.~\cite[Lemma~3.6]{GoTh} for the linear case). 
\end{proof}

Now it is clear that the action of $\pi_G(\bar L)$ on $G/H$ is polar with a section
$\pi(\bar\Sigma)$. This finishes the proof of Theorem~\ref{th:main}.
\section{The irreducibility of $(G/K,g_s)$ and its isometry group}
\label{sec:4}

The purpose of this section is to study the irreducibility of $(G/K,g_s)$ and to
compute the identity component of its isometry group. Some related results can be found in \cite{AZ}.
 \begin{proposition}
\label{pro:irr}
If $G$ is a simple Lie group, then $(G/K,g_s)$ is locally irreducible.
\end{proposition}
\begin{proof}
Let $\mathfrak{hol}(\bar G/\bar K,g_s)$ be the Lie algebra of the holonomy group of $(\bar G/\bar K,g_s)$. Since $(\bar G/\bar K,g_s)$ is a compact naturally reductive space, $\mathfrak{hol}(\bar G/\bar K,g_s)$ is the Lie subalgebra of $\mathfrak{so}(\bar{\mathfrak q})$ generated by
$\{\, \Lambda(X)\in \mathfrak{so}(\bar{\mathfrak q}) \mid X\in \bar{\mathfrak  g} \,\}$, where
$$\Lambda(X)Y=    
\begin{cases}
\mathrm{ad}(X)Y& \text{ if $X\in\bar{\mathfrak k}$,}\\
\Lambda_{\bar{\mathfrak q}}(X)Y=\frac{1}{2}[X,Y]_{\bar{\mathfrak q}}&\text{ if $X\in \bar{\mathfrak q}$;} 
\end{cases} $$
see~\cite[Theorem 4.7, p. 208]{KN}. Let $X_0\in\mathfrak k$, and $X_i$,  $Y_i\in\mathfrak p_i$, for $i=1,2$. Let $X=X_0+(0, X_2,X_2+\mathfrak k)$, $Y=Y_1+(0, s^2Y_2, (s^2-1)Y_2+\mathfrak k)$, and $Z=(0, s^2X_2,\linebreak (s^2-1)X_2+\mathfrak k)$. Then
 \begin{align}
\label{eq:1} \mathrm{ad}(X)Y& = [X_0,Y_1]+[X_2,Y_1]+(0,s^2[X_2,Y_2],(s^2-1)[X_2,Y_2]+\mathfrak k),\\
[X_1,Y]_{\bar{\mathfrak q}} &= s^2[X_1,Y_2]+(0, s^2[X_1,Y_1]_{\mathfrak p_2}, (s^2-1)[X_1,Y_1]_{\mathfrak p_2}+\mathfrak k),\label{eq:2}\\
[Z,Y]_{\bar{\mathfrak q}} &= s^2[X_2,Y_1]  \label{eq:3}\\
&\qquad +(0,s^2(2s^2-1)[X_2,Y_2],(s^2-1)(2s^2-1)[X_2,Y_2]+\mathfrak k). \nonumber
\end{align}

Now let $V$ be a $\Lambda(\bar{\mathfrak g})$-invariant subspace of $\bar{\mathfrak q}$.
It is sufficient to prove that $V=\bar{\mathfrak q}$ or  $0$. Clearly $V\cap\mathfrak p_1$
and the projection $V_{\mathfrak p_1}$ of $V$ on $\mathfrak p_1$ are
$\mathrm{ad}(\mathfrak h)$-invariant since $V$ is $\Lambda(\bar{\mathfrak k})$-invariant.
It is also clear that $G/H$ is locally irreducible since $G$ is simple. Then we have
one of the following possibilities: $V\cap\mathfrak p_1=\mathfrak p_1$ or $V\cap\mathfrak p_1=0$ and $V_{\mathfrak p_1}=0$ or $V\cap\mathfrak p_1=0$ and  $V_{\mathfrak p_1}=\mathfrak p_1$.
The first possibility would imply via \eqref{eq:2} and $[\mathfrak p_1,\mathfrak p_1]=\mathfrak h$
that $V=\bar{\mathfrak q}$. The second possibility would imply that $\{\,X_2\in\mathfrak p_2 \mid (0,s^2X_2,(s^2-1)X_2+\mathfrak k)\in V \,\}$ is an ideal of $\mathfrak g$, and so $V=0$.
Finally we consider the last possibility. Here we have that $\mathfrak k$ is an ideal of $\mathfrak g$, and hence $\mathfrak k=0$. We also have that  $\mathrm{dim} \mathfrak p_1\leq \mathrm{dim} V\leq \mathrm{dim} \mathfrak p_2$.  Since $V^\perp$ is also $\Lambda(\bar{\mathfrak g})$-invariant, $\mathrm{dim} \mathfrak p_1\leq \mathrm{dim} V^\perp\leq \mathrm{dim} \mathfrak p_2$. This implies that $\mathrm{dim} V=\mathrm{dim} \mathfrak p_1= \mathrm{dim} \mathfrak p_2$. Therefore we can write $V=\{\,X_1+(0,s^2\varphi(X_1),(s^2-1)\varphi(X_1)+\mathfrak k) \mid X_1\in\mathfrak p_1\,\}$, where $\varphi:\mathfrak p_1\to \mathfrak p_2$ is an $\mathrm{ad}(\mathfrak p_2)$-invariant isomorphism. This implies that $\mathfrak p_2$ is simple. Hence $\mathfrak p_2=[\mathfrak p_2,\mathfrak p_2]$, and from \eqref{eq:1} and \eqref{eq:3} we have that $\frac{2s^2-1}{s^2}\varphi(X_1)=\varphi(X_1)$, for all $X_1\in\mathfrak p_1$. Therefore $s=1$. It follows that $V$ cannot exist since $g_1$ is a bi-invariant metric on $G$.
\end{proof}
Let $M=G'/K'$ be a $G'$-naturally reductive space and let $\nabla^c$ be the associated canonical
connection. Let $\mathrm{I}_0(M)$ (respectively, $\mathrm{Aff}_0(\nabla^c)$) denote
the identity component of the isometry group (respectively, group of $\nabla^c$-affine transformations) of $M$. The following result \cite[p. $22$]{OlRe} can be used to compute
$\mathrm{I}_0(M)$.
\begin{proposition}
\label{pro:re}
Let $M$ be a compact locally irreducible naturally reductive space. If $M$ is neither (globally) isometric to a round sphere, nor to a real projective space, then 
$$\mathrm{I}_0(M)=\mathrm{Aff}_0(\nabla^c)=\mathrm{Tr}(\nabla^c)\hat K,$$
where $\mathrm{Tr}(\nabla^c)$ denotes the group of transvections of the associated canonical connection, and $\hat K$ denotes the connected Lie subgroup of $\mathrm{Aff}_0(\nabla^c)$ whose Lie algebra consists of $\mathrm{Tr}(\nabla^c)$-invariant fields whose associated flows are $\nabla^c$-affine. Moreover, $\mathrm{Tr}(\nabla^c)$ commutes with $\hat K$.
\end{proposition}
Next we will use Propositions \ref{pro:irr} and \ref{pro:re} to compute $\mathrm{I}_0(G/K,g_s)$.
\begin{lemma}
\label{lem:tr}
Let $\mathrm{Tr}(\nabla^c)$ be the group of transvections of the associated canonical connection to
$(G/K,g_s)$. If $G$ is simple, then $\mathrm{Tr}(\nabla^c)$ is equal to $\bar G$ for all $s\neq 1$, and equal to $G$ for $s=1$.
\end{lemma}
\begin{proof}
This follows from the fact that the Lie algebra of $\mathrm{Tr}(\nabla^c)$ is given by  
$\mathfrak{tr}(\nabla^c)=\bar{\mathfrak q}+[\bar{\mathfrak q}, \bar{\mathfrak q}]$.
\end{proof}
\begin{proposition}
\label{pro:iso1}
Let $G$ be a simple Lie group. Assume that $(G/K, g_s)$ is neither 
isometric to a round sphere, nor to a real projective space.  If
$s^2\neq 1$, then $G\times H/K$ (almost direct product) is the identity component
of the isometry group of $(G/K,g_s)$.  
\end{proposition}
\begin{proof}
By Proposition \ref{pro:re} and Lemma \ref{lem:tr} it is sufficient to prove that 
$\hat K\subset\bar G$, where $\hat K$ denotes the connected Lie subgroup of the Lie group of 
$\nabla^c$-affine transformations of $(\bar G/\bar K,g_s)$ whose Lie algebra consists of the $\bar G$-invariant fields whose associated flows
are $\nabla^c$-affine. Let $\mathfrak z(\bar{\mathfrak g})$ and 
$\mathfrak z(\mathfrak h/\mathfrak k)$ be the centers of $\bar{\mathfrak g}$ and
$\mathfrak h/\mathfrak k$ respectively. It is easy see that every 
$X\in \mathfrak z(\bar{\mathfrak g})=\mathfrak z(\mathfrak h/\mathfrak k)$ induces 
a $\bar G$-invariant field on $(\bar G/\bar K,g_s)$ (which coincides with the Killing vector field on $(\bar G/\bar K,g_s)$ induced by $X$). Conversely, let $\hat X$ be a $\bar G$-invariant field, 
and let $X\in \bar{\mathfrak q}$ such that $X.(1,1K)\bar K=\hat X_{(1,1K)\bar K}$. 
Then $X$ is a fixed vector of $\mathrm{Ad}(\bar K)$, and hence 
$\mathrm{ad}(\bar{\mathfrak k})X=0$. We write $X=X_1+(0,s^2X_2,(s^2-1)X_2+\mathfrak k)$, 
where $X_1\in\mathfrak p_1$, $X_2\in \mathfrak p_2$. Then, for all $Z_1\in\mathfrak k$ we have
\begin{align*}0&=\mathrm{ad}(Z_1) (X_1+(0,s^2X_2,(s^2-1)X_2+\mathfrak k))\\
&=[Z_1,X_1].
\end{align*}
Hence $\mathrm{ad}(\mathfrak k)X_1=0$. Analogously, for all $Z_2\in\mathfrak p_2$ we have
\begin{align*}
0&=\mathrm{ad}((0,Z_2,Z_2+\mathfrak k))(X_1+(0,s^2X_2,(s^2-1)X_2+\mathfrak k))\\
&=[Z_2,X_1] +(0,s^2[Z_2,X_2], (s^2-1)[Z_2,X_2]+\mathfrak k).
\end{align*}
Hence $\mathrm{ad}(\mathfrak p_2)X_1=0$, and $X_2$ centralizes $\mathfrak p_2$. 
Therefore $\mathrm{ad}(\mathfrak h)X_1=0$. Since $G/H$ is locally irreducible, $X_1=0$, 
and hence $X=(0,s^2X_2,(s^2-1)X_2+\mathfrak k)$, where 
$X_2+\mathfrak k\in\mathfrak z(\mathfrak h/\mathfrak k)$. Hence,
$X.(1,1K)\bar K=(0,s^2X_2,(s^2-1)X_2+\mathfrak k).(1,1K)\bar K=
(0,0,-X_2+\mathfrak k).(1,1K)\bar K$, where $X_2+\mathfrak k\in\mathfrak z(\mathfrak h/\mathfrak k)$.
\end{proof}
\begin{proposition}
\label{pro:iso2}
Let $G$ be a simple Lie group. Assume that $(G/K,g_1)$ is neither isometric to a round sphere, nor to a real projective space. If $K$ is nontrivial, then $G\times H/K$ (almost direct product) is the identity component of the isometry group of $(G/K,g_1)$.   
\end{proposition}
\begin{proof}
We first observe that $(G/K,g_1)$ is a normal homogeneous space, and by 
Lemma~\ref{lem:tr}, $\mathrm{Tr}(\nabla^c)=G=G\times 1\subset G\times H/K$.
By Proposition~\ref{pro:re}, it is
sufficient to prove that $\hat K=1\times H/K$, where $\hat K$ denotes the connected Lie subgroup of the group of $\nabla^c$-affine transformations of $(G/K,g_1)$ whose Lie algebra
consists of the $G$-invariant fields whose associated flows are $\nabla^c$-affine.
Let $F$ be the set of the fixed vectors of the action of $\mathrm{Ad}(K)$ on $\mathfrak q$.
Let $\hat{\mathfrak k}$ be the Lie algebra of $G$-invariant fields on $(G/K,g_s)$. Then
$\hat{\mathfrak k}$ is the Lie algebra of $\hat K$ and this can be
naturally identified with $F$; see~\cite{Re} for details. Since 
$\mathrm{Ad}(K)$ preserves the decomposition $\mathfrak q=
\mathfrak p_1\oplus\mathfrak p_2$ and $\mathfrak p_2\subset F$ (recall that
$[\mathfrak k,\mathfrak p_2]=0$), we have that $F=\mathfrak p_2\oplus\{X\in\mathfrak p_1 \mid
\mathrm{Ad}(K)X=X\}$. Assume that $X\in\mathfrak p_1$ is a nontrivial fixed vector of $\mathrm{Ad}(K)$. Since $H$ normalizes $K$, it preserves $F\cap\mathfrak p_1$. By irreducibility of $G/H$, $F\cap\mathfrak p_1=\mathfrak p_1$. This implies that $[\mathfrak k,\mathfrak p_1]=0$. It follows that $\mathfrak k=0$ since $G$ is simple.
This contradicts the hypothesis, and hence $F=\mathfrak p_2$. Moreover, if 
$\hat X$ is the $G$-invariant field on $G/K$ induced by $X\in \mathfrak p_2$, then the flow of
$\hat X$ is given by $\{(1, \exp{-tX}K)\}\subset 1\times H/K\subset \mathrm{I}_0(G/K,g_1)=\mathrm{Aff}_0(\nabla^c)$.  
\end{proof}

\section{Examples}
\label{sec:5}
\begin{enumerate}
\item Consider the $H$-bundle $G\to G/H$, where $G/H$ is a symmetric space of
compact type. In this case we have natural liftings of hyperpolar actions
on $G/H$ to the compact Lie group $G$ equipped with left-invariant metrics $g_s$, $s>0$.         
\item Let $G/H$ be a symmetric space of compact type, where $H$ is connected. If $H=K_1K_2$ (almost direct product) where $K_1$ and $K_2$ are connected, then we have two principal bundles $G/K_1\to G/H$ and $G/K_2\to G/H$. In this case we have natural liftings of
hyperpolar actions on $G/H$ to $(G/K_1,g_s)$ and $(G/K_2,g_s)$. It follows from Table~\ref{tab:1} that we can exhibit examples of hyperpolar actions with cohomogeneity greater than one on locally irreducible homogeneous spaces with nonnegative curvature which are not homeomorphic to symmetric spaces.
\begin{table}[ht]
\caption{}
\renewcommand\arraystretch{1.5}
\begin{tabular}{|c|c|c|}
\hline  
$G$           &$K_1$                   &$K_2$                             \\ \hline
$SU(4)$ &$SO(3)$&$SO(3)$\\ \hline
$SU(p+q)$ & $SU(p)SU(q)$ & $U(1)$                             \\ \hline
$SU(p+q)$ &$SU(p)U(1)$    & $SU(q)$                     \\ \hline
$SO(p+q)$ &$SO(p)$           &$SO(q)$                            \\ \hline
$SO(4+q)$&$SO(3)SO(q)$&$SO(3)$\\ \hline
$SO(2n)$    &$SU(n)$           & $U(1)$                             \\  \hline                   
$Sp(n)$       &$SU(n)$           &$U(1)$                               \\  \hline        
 $Sp(p+q)$  &$Sp(p)$           &$Sp(q)$      \\ \hline
 $E_6$         &$SU(6)$          &$SU(2)$          \\ \hline
$E_6$          &$Spin(10)$         &$U(1)$     \\ \hline
$E_7$          &$Spin(12)$&$SU(2)$ \\\hline
$E_7$&$E_6$&$U(1)$\\\hline
$E_8$&$E_7$&$SU(2)$\\\hline
$F_4$&$Sp(3)$&$SU(2)$\\ \hline
$G_2$&$SU(2)$&$SU(2)$\\ \hline
\end{tabular} 
\label{tab:1} 
\end{table}
\item Consider the $U(1)$-bundle $SU(n+1)/SU(n)\to SU(n+1)/S(U(n)U(1))$ and
the $Sp(1)$-bundle $Sp(n+1)/Sp(n)\to Sp(n+1)/Sp(n)Sp(1)$. By~\cite[Proposition $2.1$]{PoTh2002} and~\cite[Proposition $4.16$]{Te}, sections of polar actions on $\mathbb CP^n=SU(n+1)/S(U(n)U(1))$ and $\mathbb HP^n=Sp(n+1)/Sp(n)Sp(1)$
are totally real with respect to the complex and quaternionic structure respectively, and by ~\cite[Theorems 2.A.1 and 2.A.2]{PoTh1999}, these sections are isometric to real projective spaces. This implies that the condition $[\mathfrak m,\mathfrak m]\subset \mathfrak k$ (given in Theorem~\ref{th:main}) is always satisfied for polar actions on $\mathbb CP^n$ and $\mathbb HP^n$. Therefore, we get examples of polar actions  with nonflat sections on $(S^{2n+1}=SU(n+1)/SU(n),g_s)$ and $(S^{4n+3}=Sp(n+1)/Sp(n),g_s)$, in particular on Berger spheres.
\end{enumerate}
\section*{Acknowledgement}
The author would like to thank Claudio Gorodski for useful discussions and valuable comments.

\bibliographystyle{amsplain}
\bibliography{bibliography}
\end{document}